\documentclass[a4paper]{article}

\usepackage{fullpage}
\usepackage[utf8]{inputenc} 
\usepackage[T2A]{fontenc}

\usepackage{amsfonts}
\usepackage{amssymb, amsthm, amscd}
\usepackage{amsmath}
\usepackage{mathtools}
\usepackage{comment}
\usepackage[unicode]{hyperref}
\usepackage[usenames,dvipsnames]{xcolor}
\usepackage{colortbl}
\usepackage{textcomp}
\usepackage{cite}
\usepackage{euscript}
\usepackage[shortlabels]{enumitem}

\usepackage{graphicx}

\usepackage{float}

\newcommand{\R}{\mathbb{R}}

\newcommand{\Z}{\mathbb{Z}}

\renewcommand{\subset}{\subseteq}
\renewcommand{\phi}{\varphi}
\newcommand{\sech}{\operatorname{sech}}
\newcommand{\csch}{\operatorname{csch}}

\renewcommand{\d}{\mathrm{d}}

\theoremstyle{theorem}
\newtheorem{theorem}{Theorem}
\newtheorem{display-theorem}{Theorem}
\newtheorem{lemma}[theorem]{Lemma}
\newtheorem{proposition}[theorem]{Proposition}
\newtheorem{corollary}[theorem]{Corollary}

\theoremstyle{definition}
\newtheorem{definition}[theorem]{Definition}
\newtheorem{example}[theorem]{Example}

\theoremstyle{remark}
\newtheorem{remark}[theorem]{Remark}

\title{Logarithmic spirals on surfaces of constant Gaussian curvature}
\date{}
\author{
	Casey BLACKER
	\thanks{Department of Mathematical Sciences, George Mason University,
	\texttt{cblacke@gmu.edu}}
	,
	Pavel TSYGANENKO
	\thanks{Department of Mathematics and Computer Science, Saint Petersburg State University,
	\texttt{st104991@student.spbu.ru}}
}

\begin{document}\maketitle 

\begin{abstract}
	We compute the geodesic curvature of logarithmic spirals on surfaces of constant Gaussian curvature. In addition, we show that the asymptotic behavior of the geodesic curvature is independent of the curvature of the ambient surface. We also show that, at a fixed distance from the center of the spiral, the geodesic curvature is continuously differentiable as a function of the Gaussian curvature.
\end{abstract}

\noindent {\bf MSC classification (2020):} 53A04, 53A05, 53B20


\tableofcontents


\section*{Introduction}
A logarithmic spiral is a curve that approaches a fixed point at a constant bearing. While traditionally defined in the plane, this characterization extends naturally to any surface equipped with a metric structure. Our aim in this paper is to pursue this generalization. In particular, we first compute the geodesic curvature of logarithmic spirals on surfaces $S$ of constant Gaussian curvature $K$, and second investigate the dependence of the geodesic curvature on $K$.

The logarithmic spiral has been a source of mathematical interest at least as early as Descartes \cite{OConnorRoberston}. Jacob Bernoulli --- who styled it the \emph{spira mirabilis}, or the marvelous spiral  --- went so far as to request that the curve be inscribed on his gravestone \cite{OConnorRoberston}.

\begin{figure}
	\centering
	\includegraphics[trim={0 .7cm 0 .5cm},clip,scale=0.6]{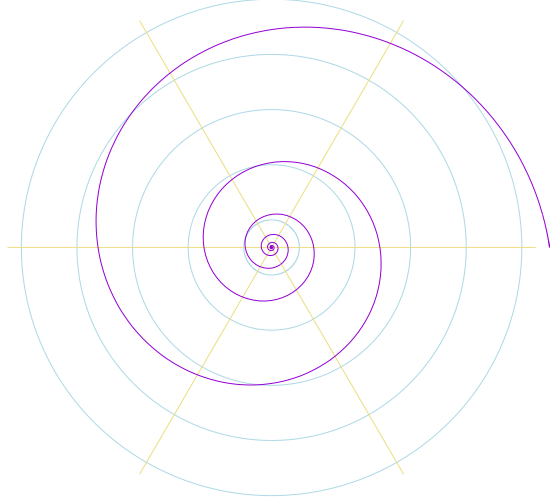}
	\caption{Logarithmic spiral}
	\label{fig:logarithmic_spiral}
\end{figure}

The paper is organized as follows.

In Section \ref{sec:background}, we review relevant background on curves and surfaces, surfaces of constant curvature, and loxodromes and logarithmic spirals.

Let $\gamma_\theta(K,r)$ denote a logarithmic spiral, meeting each member of a concentric family of circles at constant angle $0<\theta<\pi$ and parametrized by the distance $r>0$ to its center, on an oriented surface of constant Gaussian curvature $K\in\R$. In Section \ref{sec:computation_of_the_geodesic_curvature}, we show that the geodesic curvature of $\gamma_\theta$ is
\[
	k_\theta(K,r) = \cos\theta\cdot
		\begin{cases}
			\sqrt{K}\cot(r\sqrt{K})		&\text{if } K>0		\\
			1/r				&\text{if } K=0		\\
			\sqrt{-K}\coth(r\sqrt{-K})	&\text{if } K<0.
		\end{cases}
\]

In Section \ref{sec:analytic_properties_of_k}, we investigate the dependence of $k_\theta(K,r)$ on $K$. We first establish that the asymptotic behavior of $k_\theta(K,r)$ is independent of $K$. More precisely, we have the following theorem.

\setcounter{display-theorem}{24}
\begin{display-theorem}
	For all $K,K'\in\R$,
	\[
		\lim_{r\to0} \frac{k_\theta(K,r)}{k_\theta(K',r)} = 1.
	\]
\end{display-theorem}

We then show that $k_\theta(K,r)$ is differentiable in $K$ at $0$ and we compute the value of the derivative.

\setcounter{display-theorem}{26}
\begin{display-theorem}
	For every $r>0$,
	\[
		\frac{\partial k_\theta}{\partial K}(0,r) = -\frac{r}{3}\cos\theta.
	\]
\end{display-theorem}

As an immediate consequence of this, which we register in Corollary \ref{cor:k_is_continuously_differentiable_in_K}, it follows that $k_\theta(K,r)$ is continuously differentiable in $K$ wherever it is defined.

Finally, we show that near the center of the logarithmic spiral $\gamma_\theta(K,r)$ the magnitude of the geodesic curvature decreases as $K$ increases.

\setcounter{display-theorem}{29}
\begin{display-theorem}
	For all $0<\theta<\pi$, for all $K\in\R$, and for sufficiently small $r>0$,
	\[
		\frac{\partial}{\partial K}\,|k_\theta(K,r)|\leq 0,
	\]
	with equality if and only if $\theta=\frac{\pi}{2}$, that is, if and only if $\gamma_\theta(K,r)$ is a geodesic ray.
\end{display-theorem}

Loxodromes and logarithmic spirals have seen some recent interest in the literature. Properties of planar logarithmic spirals are treated in \cite{Kurnosenko09,Bolt07} and a generalization of loxodromes to hypersurfaces of revolution in Euclidean spaces is investigated in \cite{BlackwoodDukehartJavaheri17}. A very recent study of the curvature and torsion of loxodromes on surfaces of constant Gaussian curvature appears in \cite{KahramanAksoyakBektasBabaarslan23}.

In the future, it may be interesting to investigate familiar classes of curves that admit natural interpretations on arbitrary $2$-dimensional Riemannian manifolds, such as logarithmic spirals, in settings like the hyperbolic plane.

As regards notation and conventions, we broadly follow \cite{doCarmo76}.

\begin{table}[H]
	\centering
	\begin{tabular}{lll}
					&				&			\\ \hline
		symbol			&meaning			&reference		\\ \hline\vspace{-5pt}
					&				&			\\
		$S$			&surface in $\R^3$				&sec.\ \ref{subsec:curvature_definitions}\\
		$\phi(u,v)$		&parametrization of $S$				&sec.\ \ref{subsec:curvature_definitions}\\
		$\gamma_u$,$\gamma_v$	&coordinate curves				&def.\ \ref{def:coordinate_curves}	\\
		$N$			&orientation					&def.\ \ref{def:orientation}\\
		$E,F,G$			&components of the first fundamental form	&def.\ \ref{def:first_fundamental_form}\\
		$e,f,g$			&components of the second fundamental form	&def.\ \ref{def:second_fundamental_form}\\
		$K$			&Gaussian curvature				&def.\ \ref{def:Gaussian_curvature}\\
		$\gamma(t)$		&curve on a surface				&sec.\ \ref{subsec:curvature_definitions}\\
		$k$			&geodesic curvature				&def.\ \ref{def:geodesic_curvature}\\
		$r$ 			&radius of a geodesic circle			&def.\ \ref{def:geodesic_circle}\\
		$\alpha(t)$		&generating curve for a surface of revolution	&def.\ \ref{def:surface_of_revolution}\\
		$R$			&radius of a sphere or pseudosphere		&ex.\ \ref{ex:sphere}, \ref{ex:pseudosphere}\\
		$\theta$		&characteristic angle of a logarithmic spiral	&def.\ \ref{def:logarithmic_spiral}\\
		$k_1,k_2$		&geodesic curvatures of coordinate curves	&thm.\ \ref{thm:Liouville}\\
					&						&				\\
					&		&		\\
	\end{tabular}
	\caption{Glossary of notation.}
	\label{tab:symbols}
\end{table}


\section{Background}\label{sec:background}

Our first task is to establish certain background material. In Subsection \ref{subsec:curvature_definitions}, we quickly outline the basic definitions pertaining to curvature. We use this in Subsection \ref{subsec:constant_Gaussian_curvature} to show that the plane, the sphere, and the pseudosphere are, respectively, surfaces of zero, constant positive, and constant negative curvature. Finally, in Subsection \ref{subsec:loxodromes_and_logarithmic_spirals} we define loxodromes and introduce our notion of a logarithmic curve on an arbitrary surface.


\subsection{Gaussian and geodesic curvature}\label{subsec:curvature_definitions}

Before proceeding, we first recall some well-known definitions regarding curves, surfaces, and curvature. As our primary aim is to establish our notation and conventions, the exposition is brief. For a comprehensive introduction, we refer to \cite{doCarmo76}.

Fix a surface $S\subset\R^3$ and a parametrization $\phi:U\subset\R^2\to S\subset\R^3$.

\begin{definition}\label{def:orientation}
	An \emph{orientation} on $S$ is a unit normal vector field $N:S\to \R^3$.
\end{definition}

If $N$ is an orientation of $S$, then so is $-N$. Locally, the possible orientations are
\[
	N(u,v) = \pm \frac{\phi_u\times\phi_v}{\|\phi_u\times\phi_v\|}.
\]
If the surface $S$ is connected, then $S$ admits either two or zero orientations.

\begin{definition}\label{def:coordinate_curves}
	The \emph{coordinate curves} associated to a parametrization $\phi:U\subset\R^2\to S\subset\R^3$ are given by 
	\[
		\gamma_v(t) = \phi(t,v)
	\]
	and
	\[
		\gamma_u(t) = \phi(u,t),
	\]
	for fixed $v$ and $u$, respectively.
\end{definition}

\begin{definition}\label{def:first_fundamental_form}
	The \emph{first fundamental form} of $S$ at $p$ is the quadratic form $\mathrm{I}_p$ on $T_pS$ defined by
	\[
		\mathrm{I}_p (v) = \langle v,v \rangle, \quad v \in T_pS.
	\]
	If $\phi$ is the parametrization, then the coefficients of the first fundamental form are given as
	\begin{align*}
		E	&=	\langle \phi_u, \phi_u \rangle,	\\
		F	&=	\langle \phi_u, \phi_v \rangle,	\\
		G	&=	\langle \phi_v, \phi_v \rangle.
	\end{align*}
\end{definition}

\begin{definition}\label{def:second_fundamental_form}
	The \emph{second fundamental form} of $S$ at $p$ is the quadratic form $\mathrm{II}_p$ on $T_pS$ given by
	\[
		\mathrm{II}_p(v) = - \langle dN_p(v), v \rangle.
	\]
	In terms of the parametrization $\phi$, the coefficients of $\mathrm{II}$ are
	\begin{align*}
		e	&=	\langle N_u, \phi_u \rangle,	\\
		f	&=	\langle N_u, \phi_v \rangle,	\\
		g	&=	\langle N_v, \phi_v \rangle.
	\end{align*}
\end{definition}

\begin{definition}(\cite[p.\ 155, Equation 4]{doCarmo76})\textbf{.} \label{def:Gaussian_curvature}
	The Gaussian curvature of $S$ is
	\[
		K = \frac{eg - f^2}{EG - F^2}.
	\]
\end{definition}

Crucially, the Gaussian curvature does not depend either on the choice parametrization $\phi$ or the orientation $N$.

\begin{definition}\label{def:geodesic_curvature}
	The \emph{geodesic curvature} of a unit-speed curve $\alpha:I\to S$ is the quantity
	\[
		k(t) = \langle \alpha'', N \times \alpha' \rangle.
	\]
	When $k$ is constantly zero we say that $\alpha$ is a \emph{geodesic}.
\end{definition}

The geodesic curvature of $\gamma$ at $t\in\R$ depends on both the orientation of $\gamma$ (heuristically, its direction; see \cite[p.\ 6]{doCarmo76}) and that of $S$. Specifically, reparametrizing $\gamma(t)$ as $\gamma(t_0-t)$ for some $t_0\in\R$ and replacing the surface orientation $N$ with its opposite $-N$ each contribute a factor of $-1$ to $k$.

\begin{definition}
	The \emph{distance} between two points $p,q\in S$ is the infimum over the length
	\[
		\mathrm{len}(\gamma) = \int_a^b \|\gamma'(t)\| \,\mathrm{d}t
	\]
	of all paths $\gamma:[a,b]\subset\R\to S$ with $\gamma(a)=p$ and $\gamma(b)=q$, and is undefined if no such path exists.
\end{definition}

\begin{definition}\label{def:geodesic_circle}
	For $p\in S$ and $r\in(0,r_{\text{max}})$, the \emph{geodesic circle} $C_p(r)$ is the set of points $q\in S$ of distance $r$ to $p$.
\end{definition}

We will always assume $r_{\text{max}}$ to be sufficiently small to ensure that the sets $C_p(r)$ are differentiable circles in $S\subset\R^3$.

\subsection{Surfaces of constant Gaussian curvature}\label{subsec:constant_Gaussian_curvature}

Having established the basic definitions, we now turn to our three key examples of surfaces: the plane, the sphere, and the pseudosphere. After presenting their definitions, we establish that each is a surface of constant Gaussian curvature.

\begin{definition}\label{def:surface_of_revolution}
	A \emph{surface of revolution} is a surface in $\R^3$ created by rotating a curve $\alpha:\R\to\R^3$, called a \emph{profile curve} or a \emph{generating curve}, around an axis of rotation. The \emph{standard parametrization} of a surface of revolution $S$ with generating curve $\alpha(t)=(x(t),z(t))$ in the $xz$-plane and axis of rotation the $z$-axis is given by
	\[
		\phi(u,v) = (x(v)\cos u, x(v)\sin u, z(v)).
	\]
\end{definition}

\begin{example}\label{ex:plane}
	The standard parametrization of the $xy$-plane $S\subset\R^3$ as a surface of revolution with generating curve
	\[
		\alpha(t)=(t,0), \hspace{1cm}t>0,
	\]
	is the familiar polar parametrization given by
	\[
		\phi(u,v) = (v\cos u,v\sin u,0)
	\]
	for $u\in[0,2\pi)$ and $v>0$. We will always assume $S$ to be oriented by the upward-pointing unit normal vector field $N(u,v)=(0,0,1)$.
\end{example}

The Gaussian curvature of a plane $S$ is zero at every point $p\in S$. This follows from the fact that the unit normal vector field $N$ is constant: The equality $N_u=N_v=0$ implies $e=f=g=0$ from which $K=0$.

\begin{example}\label{ex:sphere}
	The parametrization of the semicircle of radius $R>0$ in the $xz$-plane,
	\[
		\alpha(t) = (R\sin t,R\cos t),	\hspace{1cm}0< t< \pi,
	\]
	yields the standard parametrization of the sphere of radius $R>0$ in $\R^3$,
	\[
		\phi(u,v) = (R\sin v\cos u, R\sin v\sin u, R\cos v).
	\]
	We will always assume that the orientation of $S$ is given by the outward pointing unit normal vector field $N(u,v)=\frac{1}{R}\,\phi(u,v)$.
\end{example}

\begin{proposition}\label{prop:sphere_Gaussian_curvature}
	The sphere of radius $R>0$ has constant Gaussian curvature $K=\frac{1}{R^2}$.
\end{proposition}

\begin{proof}
	We have
	\begin{align*}
		\phi_u	&=	(-R\sin v\sin u, R\sin v\cos u, 0)		\\
		\phi_v	&=	(R\cos v\cos u, R\cos v\sin u, -R\sin v)	\\
	\end{align*}
	and
	\begin{align*}
		N	&=	-\frac{\phi_u\times\phi_v}{\|\phi_u\times\phi_v\|}	=	(\sin v\cos u,\sin v\sin u,\cos v)		\\
		N_u	&=	(-\sin v\sin u, \sin v\cos u, 0)										\\
		N_v	&=	(\cos v\cos u,\cos v\sin u, -\sin v),
	\end{align*}
	from which
	\begin{align*}
		E	&=	\langle\phi_u,\phi_u\rangle = R^2\sin^2 v		&e &=\langle N_u,\phi_u\rangle = R\sin^2v	\\
		F	&=	\langle\phi_u,\phi_v\rangle = 0				&f &=\langle N_u,\phi_v\rangle = 0	\\
		G	&=	\langle\phi_v,\phi_v\rangle = R^2			&g &=\langle N_v,\phi_v\rangle = R,
	\end{align*}
	so that
	\[
		K = \frac{eg-f^2}{EG-F^2} = \frac{1}{R^2}.
	\]
\end{proof}

\begin{example}\label{ex:pseudosphere}
	The \emph{pseudosphere} of radius $R>0$ is the surface obtained by revolving the tractrix
	\[
		\alpha(t) = \big(R\sin t, R \ln \tan(t/2) + R\cos t\big),\hspace{1cm}0< t< \pi/2,
	\]
	in the $xz$-plane about the $z$-axis. The corresponding parametrization is
	\[
		\phi(u,v) = \big(R \sin v\cos u, R \sin v \sin u, R\ln\tan(v/2)+ R\cos v\big).
	\]
	We will orient $S$ by the outward-pointing unit vector field $N=\frac{\phi_u\times\phi_v}{\|\phi_u\times\phi_v\|}$.
\end{example}

Observe that $\phi(u,0)$ is a point at infinity, and that $\phi(u,\pi/2)$ describes a circle of radius $R$ in the $xy$-plane.

\begin{figure}
	\centering
	\includegraphics[trim={2.5cm 3cm 2.5cm 2cm},clip,scale=1]{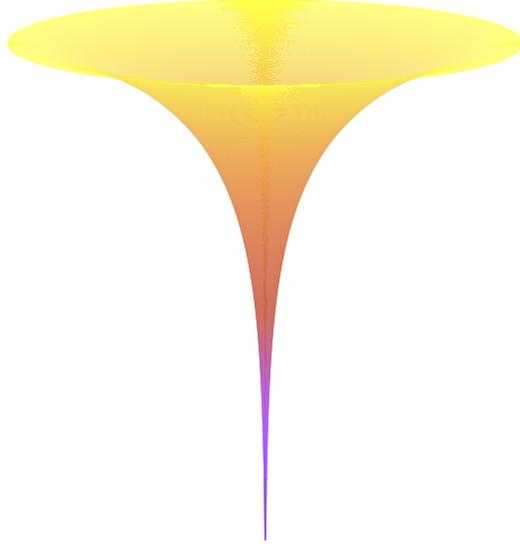}
	\caption{Pseudosphere}
	\label{fig:pseudosphere}
\end{figure}

\begin{proposition}\label{prop:pseudosphere_Gaussian_curvature}
	The pseudosphere of radius $R>0$ has constant Gaussian curvature $K=-\frac{1}{R^2}$.
\end{proposition}

\begin{proof}
	As above, we compute
	\begin{align*}
		\phi_u	&=	(-R\sin v\sin u, R\sin v\cos u, 0)		\\
		\phi_v	&=	\Big(R\cos v\cos u,R\cos v\sin u, R\frac{\sec^2(v/2)}{2\tan(v/2)} - R\sin v\Big)	\\
			&=	(R\cos v\cos u,R\cos v\sin u, R\cos v\cot v),
	\end{align*}
	where the last equality follows as
	\begin{align*}
		\frac{\sec^2(v/2)}{2\tan(v/2)} - \sin v
			&=	\frac{1}{2\sin(v/2)\cos(v/2)} - \sin v			\\
			&=	\frac{1}{\sin v} - \frac{\sin^2 v}{\sin v}		\\
			&=	\cos v\cot v.
	\end{align*}
	Consequently,
	\begin{align*}
		N	&=	\frac{\phi_u\times\phi_v}{\|\phi_u\times\phi_v\|}	
			 =	(\cos u\cos v, \sin u\cos v,-\sin v)		\\
		N_u	&=	(-\sin u\cos v, \cos u\cos v,0)		\\
		N_v	&=	(-\cos u\sin v, -\sin u\sin v,-\cos v),
	\end{align*}
	and so
	\begin{align*}
		E	&=	\langle\phi_u,\phi_u\rangle = R^2\sin^2v		&e &=\langle N_u,\phi_u\rangle = R\sin v\cos v	\\
		F	&=	\langle\phi_u,\phi_v\rangle = 0				&f &=\langle N_u,\phi_v\rangle = 0			\\
		G	&=	\langle\phi_v,\phi_v\rangle = R^2\cos^2v\csc^2v		&g &=\langle N_v,\phi_v\rangle = -R\cos v\csc v.
	\end{align*}
	We conclude that
	\[
		K = \frac{eg-f^2}{EG-F^2} = -\frac{1}{R^2}.
	\]
\end{proof}


\subsection{Loxodromes and logarithmic spirals}\label{subsec:loxodromes_and_logarithmic_spirals}
 
In this subsection, we define loxodromes on surfaces of revolution and logarithmic spirals on arbitrary surfaces. We examine explicit parametrizations of these curves on planes and spheres.

\begin{definition}
	Let $S\subset\R^3$ be a surface of revolution with generating curve $\alpha(t)=(x(t),z(t))$.
	\begin{enumerate}
		\item A \emph{parallel} is a circle in $S$ obtained by rotating a single point $p\in S$ about the axis of rotation.
		\item A \emph{meridian} is a curve in $S$ obtained by rotating the generating curve $\alpha$ by a fixed angle $u\in\R$ about the axis of rotation.
		\item A \emph{loxodrome} is a curve that intersects all parallels at a fixed angle $0<\theta<\pi$.
	\end{enumerate}
\end{definition}

\begin{remark}
	Parallels and meridians are limiting cases of loxodromes as $\theta$ approaches $0$ and $\pi$, respectively. In terms of the standard parametrization $\phi$, the parallels of $S$ are the coordinate curves
	\[
		\gamma_v(t) = \phi(t,v)
	\]
	and the meridians are the coordinate curves
	\[
		\gamma_u(t) = \phi(u,t),
	\]
	for fixed values $u,v\in\R$.
\end{remark}

Note that the convention that $\theta$ is the signed angle from $\gamma_v'$ to $\gamma'$, with respect to the orientation $N$ of the underlying surface, determines an orientation on $\gamma$. We will always assume $\gamma$ to be equipped with this orientation, regardless of our choice of parametrization.

\begin{definition}\label{def:logarithmic_spiral}
	A \emph{logarithmic spiral} with characteristic angle $0<\theta<\pi$ on a surface $S$ about a fixed point $p \in S$ is a curve $\gamma(t)$ that intersects each geodesic circle $C_p(r)$ at constant angle $\theta$.
\end{definition}

We will adopt the convention that a logarithmic spiral is oriented so that $\theta$ is the signed angle from $\gamma_{C_p(r)}'$ to $\gamma'$, with respect to the orientation of the underlying surface, where $\gamma_{C_p(r)}(t)$ is a positively-oriented parametrization of $C_p(r)$. As with loxodromes, we will assume $\gamma$ to carry this orientation irrespective of our choice of parametrization.

Since the parallels $\gamma_v(t)$ on a sphere are the geodesic circles about the north and south poles, it follows that every loxodrome on a sphere is a logarithmic spiral about the north and south poles. Indeed, the logarithmic spirals on a sphere or a plane are precisely the loxodromes with respect to a suitable parametrization.

In contrast, loxodromes and logarithmic spirals on a pseudosphere are distinct. This is easily seen, as the logarithmic spirals necessarily approach their center point, while the loxodromes on the pseudosphere do not approach any point. To clarify this qualitative difference, a visual approximation to a loxodrome on a pseudosphere is provided in Figure \ref{fig:pseudosphere_loxodrome}.

\begin{figure}
	\centering
	\includegraphics[trim={2.5cm 3cm 2.5cm 2cm},clip,scale=1]{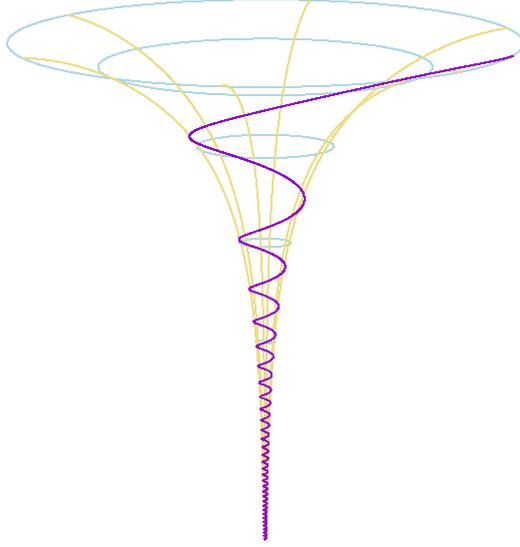}
	\caption{Loxodrome on a pseudosphere}
	\label{fig:pseudosphere_loxodrome}
\end{figure}

\begin{proposition}\label{prop:classic_logarithmic_curve}
	For $a\neq 0$, the curve
	\[
		\gamma(t) = e^{-at}(\cos t,\sin t),
	\]
	is a logarithmic spiral with characteristic angle $\theta=\arctan a$ about the origin in the plane $\R^2$.
\end{proposition}

\begin{proof}
	Note that $\sin\theta(t)$ is equal to the dot product of the unit tangent vector $\gamma'(t)/\|\gamma'(t)\|$ and the inward-pointing unit radial vector $-\gamma(t)/\|\gamma(t)\|$. Thus, from
	\begin{align*}
		\gamma'		&=	e^{-at}(-a\cos t-\sin t,-a\sin t+\cos t)		\\
		\|\gamma'\|	&=	e^{-at}\sqrt{1+a^2}
	\end{align*}
	we obtain
	\[
		\sin\theta = \Big\langle\frac{\gamma'}{\|\gamma'\|},-\frac{\gamma}{\|\gamma\|}\Big\rangle
			= \frac{a}{\sqrt{1+a^2}}
	\]
	so that $\theta = \arctan a$.
\end{proof}

\begin{proposition}\label{prop:loxodrome}
	Consider the sphere $S$ of radius $R>0$ centered at the origin $(0,0,0)\in\R^3$. For $a\neq 0$, the curve
	\[
		\gamma(t) = R\,\big(\sin(2t)\cos(a\ln\tan t), \sin(2t)\sin(a\ln\tan t), \cos(2t)),	\hspace{1cm}0<t<\frac{\pi}{2},
	\]
	is a loxodrome with characteristic angle $\theta=\operatorname{arccot}a$.
\end{proposition}

\begin{proof}
	First observe that $\gamma(t) = \phi\big(a\ln\tan t,2t\big)$, where
	\[
		\phi(u,v) = R\,(\sin v\cos u, \sin v\sin u,\cos v)
	\]
	is the standard parametrization of the sphere $S$ (see Example \ref{ex:sphere}).

	Since $\gamma(t)$ spirals outwards from the north pole, the quantity $\sin\theta(t)$ is the dot product of the unit tangent vector $\gamma'(t)/\|\gamma'(t)\|$ and the southward-pointing unit radial vector $\phi_v/\|\phi_v\|$ at $\gamma(t)$. Direct computations yield
	\begin{align*}
		\gamma'		&=	2R\,\big(\cos(2t)\cos(a\ln\tan t) - a\sin(a\ln\tan t),	\\
				&\hspace{1.2cm}	\cos(2t)\sin(a\ln\tan t)+a\cos(a\ln\tan t),	\\
				&\hspace{1.2cm}	-\sin(2t)\hspace{.2pt}\big)			\\
		\|\gamma'\|	&=	2R\,\sqrt{1+a^2}	
	\end{align*}
	and
	\begin{align*}
		\phi_v(a\ln\tan t,2t)		&=	R\,\big(\cos(2t)\cos(a\ln\tan t), \cos(2t)\sin(a\ln\tan t),-\sin(2t)\big)		\\
		\|\phi_v(a\ln\tan t,2t)\|	&=	R.
	\end{align*}
	Thus,
	\[
		\sin\theta = \Big\langle\frac{\gamma'}{\|\gamma'\|}, \frac{\phi_v}{\|\phi_v\|}\Big\rangle = \frac{1}{\sqrt{1+a^2}}
	\]
	and we conclude that $\theta=\operatorname{arccot} a$.
\end{proof}

Note that in Proposition \ref{prop:loxodrome} the curve $\gamma$ is a logarithmic spiral about both the north and the south poles.

It is readily seen that every logarithmic spiral on the plane and the sphere with characteristic angle $\theta$ is the image of that given in Propositions \ref{prop:classic_logarithmic_curve} and \ref{prop:loxodrome}, respectively, by an isometry of the underlying space. Similarly, the planar and spherical loxodromes are precisely the rotations of these examples about the origin and the $z$-axis, respectively.

\begin{figure}
	\centering
	\includegraphics[trim={2.5cm 3cm 2.5cm 2.7cm},clip,scale=1]{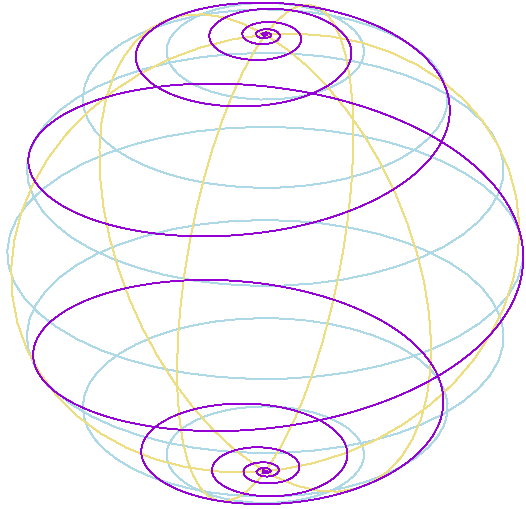}
	\caption{Loxodrome on a sphere}
	\label{fig:loxodrome}
\end{figure}


\section{Computation of the geodesic curvature}\label{sec:computation_of_the_geodesic_curvature}

Our aim in this section is to explicitly compute the geodesic curvature of logarithmic spirals based at $p\in S$ when $S$ has zero, positive, and negative constant Gaussian curvature, respectively. This corresponds to the setting of planes (and cylinders), spheres, and pseudospheres.

Our approach is to first compute the geodesic curvature of meridians and geodesic circles on surfaces of revolution $S$ and then to apply Liouville's theorem, as follows.

\begin{theorem}[Liouville's formula, \cite{doCarmo76} p.\ 253]\label{thm:Liouville}
	If $\gamma(t):I\to S$ is unit-speed, and if the parametrization $\phi(u,v)$ satisfies $\langle\phi_u,\phi_v\rangle=0$, then
	\[
		k = k_1\cos\theta + k_2\sin\theta + \frac{\d\theta}{\d t},
	\]
	where $k$, $k_1$ and $k_2$ are the geodesic curvatures of $\gamma(t)$, $\gamma_v(t)=\phi(t,v)$, and $\gamma_u(t)=\phi(u,t)$, respectively, and where $\theta(t)$ is the time-varying angle from  $\gamma_v'(t)$ to $\gamma'(t)$.
\end{theorem}


\subsection{Planes}

We begin with the traditional planar logarithmic spirals.

\begin{proposition}\label{prop:logarithmic_spiral_curvature}
	Let $S\subset\R^3$ be the $xy$-plane, oriented by the upward-pointing unit normal vector field. The geodesic curvature of a logarithmic spiral with characteristic angle $\theta$ about any point $p\in S$ is
	\[
		k = \frac{1}{r} \cos\theta,
	\]
	where $r(t)$ is the distance from $\gamma(t)$ to $p$.
\end{proposition}

\begin{proof}
	For ease of exposition, suppose for the moment that $p$ is the origin $(0,0,0)\in\R^3$. Consider the standard polar parametrization of $S$ given by
	\[
		\phi(u,v) = (v\cos u,v\sin u,0).
	\]
	Since $\langle\phi_u,\phi_v\rangle=0$, Theorem \ref{thm:Liouville} yields
	\[
		k = k_1 \cos\theta + k_2 \sin\theta + \frac{\d\theta}{\d t},
	\]
	where $k_1$ represents the geodesic curvatures of the circles
	\[
		\gamma_v(t) = (v\cos t,v\sin t,0),	\hspace{1cm}v>0,
	\]
	and $k_2$ represents the the geodesic curvatures of the rays
	\[
		\hspace{.8cm}\gamma_u(t) = (t\cos u,t\sin u,0),	\hspace{1cm}u\in[0,2\pi).
	\]
	From $\|\gamma_v'\|=v$ and $\|\gamma_u'\|=1$ it follows that $\gamma_v(t/v)$ and $\gamma_u(t)$ are unit-speed curves. Straightforward computations yield
	\[
		k_1(t/v)	=	\big\langle\gamma_v(t/v)'',N(t/v)\times\gamma_v(t/v)'\big\rangle = \frac{1}{v}
	\]
	and
	\[
		k_2(t) = \big\langle \gamma_u''(t), N(t)\times\gamma_u'(t) \big\rangle = 0,
	\]
	where $N=(0,0,1)$. Since $\theta(t)$ is constant for any logarithmic spiral, it follows that $\frac{\d\theta}{\d t}=0$. Theorem \ref{thm:Liouville} and the identity $v=r$ yield
	\[
		k = \frac{1}{r} \cos\theta.
	\]
	In the case that $p$ is not the origin, we obtain a nearly identical computation with respect to the polar parametrization $\psi(u,v) = \phi(u,v)-p$ centered at $p$.
\end{proof}


\subsection{Spheres}

We will consider a loxodrome on a sphere of radius $R$. Similarly to the case of a logarithmic spiral, the geodesic curvature of the loxodrome will be the geodesic curvature of the parallel multiplied by the cosine of the angle $\theta$ between the spiral and the parallels.  

\begin{proposition}\label{prop:sphere_loxodrome_curvature}
	The geodesic curvature of a loxodrome $\gamma(t)$ on a sphere $S$ of radius $R>0$, oriented by the outward-pointing unit normal vector field, with defining angle $\theta$ is
	\[
		k(r) = \frac{1}{R} \cot\frac{r}{R} \cos\theta,
	\]
	where $r$ is the distance from the north pole of the sphere to a parallel.
\end{proposition}

\begin{proof}
	Recall the standard parametrization of $S$,
	\[
		\phi(u,v) = R\,(\sin v\cos u,\sin v\sin u,\cos v).
	\]
	Since $\langle\phi_u,\phi_v\rangle=0$, the conditions of Theorem \ref{thm:Liouville} apply. First note that the parallels
	\[
		\gamma_v(t) = R\,(\sin v\cos t,\sin v\sin t,\cos v)
	\]
	and the meridians
	\[
		\gamma_u(t) = R\,(\sin t\cos u,\sin t\sin u,\cos t)
	\]
	satisfy
	\[
		\|\gamma_v'\|=R\sin v,\hspace{1.1cm}	\|\gamma'_u\|=R.
	\]
	We deduce that $\gamma_v(t/R\sin v)$ and $\gamma_u(t/R)$ are their respective unit-speed reparametrizations. As
	\[
		N(p) = p/R
	\]
	for all $p\in S$, we obtain
	\begin{align*}
		k_1(t/R\sin v)	&=	\big\langle\gamma_v(t/R\sin v)'',N(t/R\sin v)\times\gamma_v(t/R\sin v)'\big\rangle = \frac{1}{R}\cot v		\\
		k_2(t/R)	&=	\big\langle\gamma_u(t/R)'',N(t/R)\times\gamma_u(t/R)'\big\rangle = 0.
	\end{align*}
	As $\theta(t)$ is constant, an application of Theorem \ref{thm:Liouville} gives
	\[
		k = \frac{1}{R}\cot v\cos\theta = \frac{1}{R}\cot\frac{r}{R}\cos\theta.
	\]
\end{proof}


\subsection{Pseudospheres}

We now consider logarithmic spirals on a pseudosphere of radius $R>0$. Recall from Example \ref{ex:pseudosphere} that the pseudosphere is obtained by rotating the tractrix
\[
	\alpha(t) = R\,\big(\sin t, \ln\tan(t/2) + \cos t\big),\hspace{1cm}0< t\leq \pi/2,
\]
in the $xz$-plane about the $z$-axis, with associated parametrization
\[
	\phi(u,v) = R\,\big(\sin v\cos u,\sin v\sin u,\ln\tan(v/2)+\cos v\big).
\]

We will consider a classical loxodrome on a pseudosphere that intersects the parallels $z=\text{const.\ }$ at a constant angle. Unlike the previous cases, loxodromes on a pseudosphere are not logarithmic spirals. Thus, we must consider them separately. 

\begin{proposition}\label{prop:pseudophere_loxodrome_curvature}
	The geodesic curvature of a loxodrome with characteristic angle $\theta$ on a pseudosphere of radius $R>0$, oriented by the outward-pointing unit normal vector field, is
	\[
		k= -\frac{1}{R}\cos\theta.
	\]
	In particular, the geodesic curvature is constant.
\end{proposition}

\begin{proof}
	Our approach, as above, is to apply Theorem \ref{thm:Liouville} with respect to the standard parametrization
	\[
		\phi(u,v) = R\,\big(\sin v\cos u,\sin v\sin u,\ln\tan(v/2)+\cos v\big).
	\]
	The parallels and meridians are, respectively,
	\begin{align*}
		\gamma_v(t) = R\,\big(\sin v\cos t,\sin v\sin t,\ln\tan(v/2)+\cos v\big)	\\
		\gamma_u(t) = R\,\big(\sin t\cos u,\sin t\sin u,\ln\tan(t/2)+\cos t\big).
	\end{align*}
	This gives
	\begin{align*}
		\gamma_v'(t) &= R\,\big(-\sin v\sin t,\sin v\cos t,0)		\\
		\gamma_u'(t) &= R\,\big(\cos t\cos u,\cos t\sin u,\cos t\cot t),
	\end{align*}
	so that
	\[
		\|\gamma_v'(t)\| = R\sin v,\hspace{1.1cm}	\|\gamma_u'(t)\| = R\cot t.
	\]
	Thus, $\gamma_v(t/R\sin v)$ is a unit-speed reparametrization of $\gamma_v(t)$. Taking
	\[
		N(u,v) = \frac{\phi_u\times\phi_v}{\|\phi_u\times\phi_v\|} = (\cos v\cos u,\cos v\sin u,-\sin v)
	\]
	yields
	\begin{align*}
		k_1(t/R\sin v)
			&=	\big\langle \gamma_v(t/R\sin v)'', N(t/R\sin v, v)\times \gamma_v(t/R\sin v)'\big\rangle			\\
			&=	\Big\langle -\frac{1}{R\sin v}\big(\cos(t/R\sin v), \sin(t/R\sin v),0\big), 					\\
			&		\hspace{1.8cm} \big(\sin v \cos(t/R\sin v), \sin v\sin(t/R\sin v), \cos v\big)\Big\rangle		\\
			&=	-\frac{1}{R}.
	\end{align*}
	Rather than perform a computation, we will here invoke the fact (see \cite[p.\ 255]{doCarmo76}) that the meridians $\gamma_u$ of any surface of revolution are geodesics, and thus in particular that $k_2=0$. By Theorem \ref{thm:Liouville}, we conclude that
	\[
		k = -\frac{1}{R}\cos\theta.
	\]
\end{proof}

Let us consider a logarithmic spiral around a point $p$ on a pseudosphere $S$.

\begin{proposition}\label{prop:pseudosphere_logarithmic_spiral_curvature}
	The geodesic curvature of a logarithmic spiral with characteristic angle $\theta$ on a pseudosphere of radius $R>0$, oriented by the outward-pointing unit normal vector field, is
	\[
		k(r) = \frac{1}{R} \coth{\frac{r}{R}} \cos{\theta},
	\]
	where $r$ is the distance from $p$. In particular, $k$ is constant.
\end{proposition}

\begin{proof}
	Let us introduce polar coordinates $(r,u)$ with pole $p\in S$. Specifically, $r$ is the distance to $p$, and $u$ is an angular coordinate as above. The reader is advised that, to conform with do Carmo's convention for geodesic polar coordinates \cite[p.\ 286]{doCarmo76}, and in contrast with our convention for surfaces of revolution, the angular coordinate $u$ is now second in the system $(r,u)$.

	It is shown in \cite[p.\ 287]{doCarmo76} that
	\[
		E = 1, \qquad F = 0, \qquad \lim_{r\to 0}{G} = 0, \qquad \lim_{r \to 0}{(\sqrt{G})_r} = 1.
	\]
	In particular, the equality $F=0$ implies that we may apply Theorem \ref{thm:Liouville}. Let us now find expressions for $k_1$, $k_2$, and $\frac{\d\theta}{\d t}$.

	In \cite[p.\ 254]{doCarmo76} it is shown that for \emph{any} orthogonal parametrization $(r,u)$ we have 
	\[
		k_1 = \frac{G_r}{2G\sqrt{E}}.
	\]
	In particular, this is the value of the geodesic curvature of the geodesic circle $C_r(p)$. Further observe that the \emph{radial geodesics} (see \cite[p.\ 286]{doCarmo76}), obtained by fixing the angular coordinate $u$, are, by construction, geodesic curves and as such satisfy $k_2=0$. Finally, note that $\frac{\d\theta}{\d t}=0$, since $\theta(t)$ is constant for any logarithmic spiral.

	Our aim now is to determine $G(r,u)$ and $E(r,u)$. By the differential equation \cite[p.\ 288]{doCarmo76}
	\[
		(\sqrt{G})_{rr} + K \sqrt{G} = 0,
	\]
	in conjunction with Proposition \ref{prop:pseudosphere_Gaussian_curvature}, we arrive at
	\[
		\sqrt{G}(r) = c_1 \cosh\frac{r}{R} + c_2 \sinh\frac{r}{R}.
	\]
	Thus we obtain
	\begin{align*}
		c_1	&=	\lim_{r\to0} \:\Big(c_1 \cosh\frac{r}{R} + c_2 \sinh\frac{r}{R}\Big) = \lim_{r\to0} \sqrt{G} = 0			\\
		c_2	&=	R\;\lim_{r\to0} \,\frac{1}{R}\Big(c_1 \sinh\frac{r}{R} + c_2 \cosh\frac{r}{R}\Big) = R\,\lim_{r\to0} (\sqrt{G})_r = R,
	\end{align*}
	and it follows that $G(r) = R^2\sinh^2(r/R)$. Consequently,
	\[
		k_1 =\frac{G_r}{2G} = \frac{2R \sinh(r/R)\cosh(r/R)}{2 R^2\sinh^2(r/R)} = \frac{1}{R}\coth\frac{r}{R},
	\]
	and we conclude by Theorem \ref{thm:Liouville} that
	\[
		k = \frac{1}{R} \coth\frac{r}{R} \cos\theta.
	\]
\end{proof}

Note that we could use the approach of Proposition \ref{prop:pseudosphere_logarithmic_spiral_curvature} to establish Propositions \ref{prop:logarithmic_spiral_curvature} and \ref{prop:sphere_loxodrome_curvature}. Indeed, this method is more powerful insofar as it immediately establishes the natural analogues of Propositions \ref{prop:logarithmic_spiral_curvature}, \ref{prop:sphere_loxodrome_curvature}, and \ref{prop:pseudosphere_logarithmic_spiral_curvature} with respect to any surface of constant curvature. In this exposition, we have taken a more concrete route.


\section{Analytic properties of $k_\theta(K,r)$}\label{sec:analytic_properties_of_k}

Fix an angle $0<\theta<\pi$ and let $\gamma_\theta(K,r)$ be a logarithmic spiral with characteristic angle $\theta$ on a surface of constant Gaussian curvature $K$, parametrized by the distance $r>0$ to its center point, and let $k_\theta(K,r)$ be the geodesic curvature of $\gamma_\theta(K,r)$.

In Propositions \ref{prop:logarithmic_spiral_curvature}, \ref{prop:sphere_loxodrome_curvature} and \ref{prop:pseudosphere_logarithmic_spiral_curvature}, we computed $k_\theta(K,r)$ in terms of $r$ and $K$. In Propositions \ref{prop:sphere_Gaussian_curvature} and \ref{prop:pseudosphere_Gaussian_curvature}, we determined the Gaussian curvature $K$ in terms of $R$. Putting everything together, we have
\[
	k_\theta(K,r) = \cos\theta\cdot
		\begin{cases}
			\sqrt{K}\cot(r\sqrt{K})		&\text{if } K>0		\\
			1/r				&\text{if } K=0		\\
			\sqrt{-K}\coth(r\sqrt{-K})	&\text{if } K<0
		\end{cases}
\]
Since the maximum distance between any two points on a sphere of radius $R>0$ is $\pi R$, it follows that if $K>0$ then $k_\theta(K,r)$ is defined only when $r<\pi/\sqrt{K}$.

\begin{figure}
	\centering
	\includegraphics[trim={.4cm 1cm 1.3cm 1.9cm},clip,scale=1]{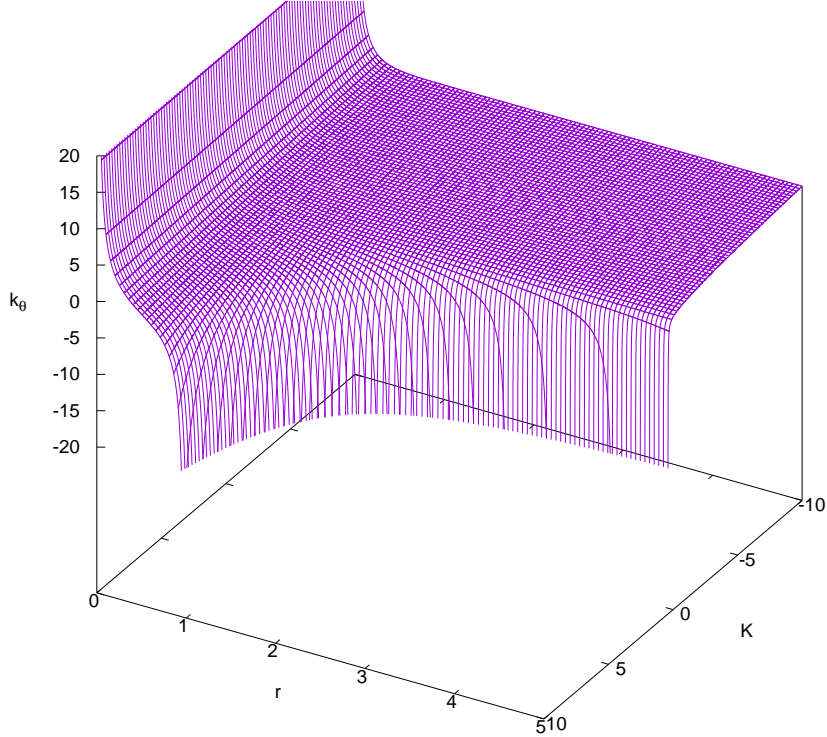}
	\caption{The function $k_\theta(K,r)$}
	\label{fig:graph_of_k}
\end{figure}

Our aim in this section is to study the function $k_\theta(K,r)$. We first show that the asymptotic behavior of $k_\theta(K,r)$ as $r$ approaches $0$ is independent of $K$. This aligns well with our intuitions, since any Riemannian structure on a smooth surface infinitesimally approximates that of the Euclidean plane.

\begin{theorem}\label{thm:short_time_behavior_of_k}
	For all $K,K'\in\R$,
	\[
		\lim_{r\to0} \frac{k_\theta(K,r)}{k_\theta(K',r)} = 1.
	\]
\end{theorem}

\begin{proof}
	If $K>0$, then
	\begin{align*}
		\lim_{r\to0} \frac{k_\theta(0,r)}{k_\theta(K,r)}
		&=	\lim_{r\to0}\frac{1/r}{\sqrt{K}\cot(r\sqrt{K})}				\\
		&=	\lim_{r\to0}\frac{\tan(r\sqrt{K})}{r\sqrt{K}}				\\
		&=	\lim_{r\to0}\;\sec^2(r\sqrt{K})\\
		&=	1.
	\end{align*}
	where we have used l'H{\^o}pital's rule \cite[p.\ 109]{Rudin76} in the third equality. Similarly, if $K<0$, then
	\begin{align*}
		\lim_{r\to0} \frac{k_\theta(0,r)}{k_\theta(K,r)}
		&=	\lim_{r\to0}\frac{1/r}{\sqrt{-K}\coth(r\sqrt{-K})}				\\
		&=	\lim_{r\to0}\frac{\tanh(r\sqrt{-K})}{r\sqrt{-K}}				\\
		&=	\lim_{r\to0}\;\sech^2(r\sqrt{-K})\\
		&=	1.
	\end{align*}
	We conclude that
	\[
		\lim_{r\to0} \frac{k_\theta(K,r)}{k_\theta(K',r)} = \bigg(\lim_{r\to0} \frac{k_\theta(0,r)}{k_\theta(K,r)}\bigg)^{-1} \bigg(\lim_{r\to0}\frac{k_\theta(0,r)}{k_\theta(K',r)}\bigg)= 1.
	\]
\end{proof}

We now turn to investigate $k_\theta(K,r)$ as a function of the Gaussian curvature $K\in\R$.

\begin{lemma}\label{lem:k_is_continuous_in_K}
	The function $k_\theta(K,r)$ is continuous in $K$ on its domain of definition.
\end{lemma}

\begin{proof}
	Fix $r>0$. We have
	\[
		\lim_{K\to0^+} \frac{k_\theta(K,r)}{\cos\theta} = \lim_{K\to0^+} \frac{\sqrt{K}}{\tan(r\sqrt{K})} = \lim_{K\to0} \frac{1}{r\sec^2(r\sqrt{K})} = \frac{1}{r}
	\]
	and
	\[
		\lim_{K\to0^-} \frac{k_\theta(K,r)}{\cos\theta} = \lim_{K\to0^+} \frac{\sqrt{K}}{\tanh(r\sqrt{K})} = \lim_{K\to0} \frac{1}{r\sech^2(r\sqrt{K})} = \frac{1}{r}
	\]
	where in each instance we invoke l'H{\^o}pital's rule in the second equality. Since $k_\theta(0,r)=\cos\theta\cdot 1/r$, it follows that $k_\theta(K,r)$ is continuous in $K$ at $0$.
\end{proof}

For fixed $r>0$, define the auxiliary function
\[
	f(t) = 
		\begin{cases}
			\sqrt{t}\cot(r\sqrt{t})		&\text{if } t>0		\\
			1/r				&\text{if } t=0		\\
			\sqrt{-t}\coth(r\sqrt{-t})	&\text{if } t<0,
		\end{cases}
\]
and observe that its derivative is
\[
		f'(t) = \frac{1}{2\sqrt{|t|}}\cdot
			\begin{cases}
				\phantom{-}\cot(r\sqrt{t}) -r\sqrt{t}\csc^2(r\sqrt{t})	&\text{if } t>0		\\
				-\coth(r\sqrt{-t})+r\sqrt{-t}\csch^2(r\sqrt{-t})	&\text{if } t<0.
			\end{cases}
\]
From the identity $k_\theta(K,r)=\cos\theta\cdot f(K)$, we see that for admissible values of $r$ and $K$, the quantity $f(K)$ represents the geodesic curvature of the positively-oriented geodesic circle of radius $r$ concentric with the spiral $\gamma_\theta(K,r)$.

\begin{theorem}\label{thm:derivative_of_k_at_0}
	For every $r>0$,
	\[
		\frac{\partial k_\theta}{\partial K}(0,r) = -\frac{r}{3}\cos\theta.
	\]
\end{theorem}

\begin{proof}
	We will show that $f'(0) = -\frac{1}{3}r$. Direct computations yield 
	\begin{align*}
		\lim_{t\to0^+}\,f'(t)
			&=	r\lim_{t\to0^+} \frac{\cot(r\sqrt{t}) - r\sqrt{t}\csc^2(r\sqrt{t})}{2r\sqrt{t}}	\\
			&=	r\lim_{a\to0} \frac{\sin a\cos a - a}{2a\sin^2a}					\\
			&=	r\lim_{a\to0} \,\frac{\cos^2a-\sin^2a-1}{2\sin^2a+4a\sin a\cos a}		\\
			&=	r\lim_{a\to0} \,\frac{-\sin^2a}{\sin^2a+2a\sin a\cos a}				\\
			&=	r\lim_{a\to0}\,\frac{-\sin a}{\sin a+2a\cos a}					\\
			&=	r\lim_{a\to0}\, \frac{-\cos a}{3\cos a-2a\sin a}				\\
			&=	-\frac{r}{3}
	\end{align*}
	and
	\begin{align*}
		\lim_{t\to0^-}\,f'(t)
			&=	r\lim_{t\to0^+} \frac{r\sqrt{t}\csch^2(r\sqrt{t}) - \coth(r\sqrt{t})}{2r\sqrt{t}}	\\
			&=	r\lim_{a\to0} \frac{a-\sinh a\cosh a}{2a\sinh^2a}					\\
			&=	r\lim_{a\to0} \,\frac{1-\cosh^2a-\sinh^2a}{2\sinh^2a+4a\sinh a\cosh a}			\\
			&=	r\lim_{a\to0} \,\frac{-\sinh^2a}{\sinh^2a+2a\sinh a\cosh a}				\\
			&=	r\lim_{a\to0} \,\frac{-\sinh a}{\sinh a+2a\cosh a}					\\
			&=	r\lim_{a\to0}\,\frac{-\cosh a}{3\cosh a+2a\sinh a}					\\
			&=	-\frac{r}{3},
	\end{align*}
	where in each case we use l'H{\^o}pital's rule in the third and sixth equalities, and where $a=r\sqrt{t}$. Hence
	\[
		\lim_{t\to0} f'(t) = -\frac{r}{3}.
	\]
	Lemma \ref{lem:k_is_continuous_in_K} ensures that
	\[
		\lim_{t\to0}\:f(t)-f(0) = 0,
	\]
	and thus a further application of l'H{\^o}pital's rule yields
	\[\label{eq:derivative_of_k_at_0}\tag{$*$}
		f'(0) = \lim_{t\to 0}\frac{f(t)-f(0)}{t} = \lim_{t\to0} f'(t) = -\frac{r}{3}.
	\]
	The result follows as $k_\theta(K,r)=\cos\theta\cdot f(K)$.
\end{proof}

\begin{corollary}\label{cor:k_is_continuously_differentiable_in_K}
	The function $k_\theta(K,r)$ is continuously differentiable in $K$, wherever it is defined.
\end{corollary}

\begin{proof}
	Fix $r>0$. Equation (\ref{eq:derivative_of_k_at_0}) in the proof of Theorem \ref{thm:derivative_of_k_at_0} implies that
	\[
		\lim_{K\to0}\frac{\partial k_\theta}{\partial K}(K,r) = \frac{\partial k_\theta}{\partial K}(0,r)
	\]
	and we conclude by noting that $\frac{\partial k_\theta}{\partial K}(K,r)$ is clearly continuous at every $K\neq0$ at which $k_\theta(K,R)$.
\end{proof}

\begin{lemma}\label{lem:f_is_decreasing}
	$f'(t)<0$ for every $t\in\R$ at which $f(t)$ is defined.
\end{lemma}

\begin{proof}
	First suppose $t>0$. From the inequality
	\[
		(\sin a\cos a)' = \cos^2 a - \sin^2 a \;\;<\;\; \cos^2 a + \sin^2 a = a',	\hspace{1cm}a\notin\pi\Z,
	\]
	and the identity
	\[
		\sin a\cos a = a,	\hspace{1cm}\text{when }a=0,
	\]
	it follows that
	\[
		\sin a\cos a < a,	\hspace{1cm}\text{for all }a>0.
	\]
	Dividing through by $\sin^2 a$ yields
	\[
		\cot a < a \csc^2 a,	\hspace{1cm}a\notin \pi\Z,
	\]
	and the substitution $a=r\sqrt{t}$ gives
	\[
		f'(t) = \frac{1}{2\sqrt{|t|}}\cdot \big(\cot(r\sqrt{t}) -r\sqrt{t}\csc^2(r\sqrt{t})\big) < 0,	\hspace{1cm}r\sqrt{t}\notin \pi\Z.
	\]

	When $t<0$, the result follows from
	\[
		(\sinh a\cosh a)' = \cosh^2 a + \sinh^2 a \;\;>\;\; \cosh^2 a - \sinh^2 a = a',	\hspace{1cm}a\neq 0,
	\]
	by a similar derivation. Finally, we established that $f'(0)=-\frac{1}{3}r<0$ in the course of the proof of Theorem \ref{thm:derivative_of_k_at_0}.
\end{proof}

\begin{theorem}\label{thm:norm_of_variation_of_curvature}
	For all $0<\theta<\pi$, for all $K\in\R$, and for sufficiently small $r>0$,
	\[
		\frac{\partial}{\partial K}\,|k_\theta(K,r)|\leq 0,
	\]
	with equality if and only if $\theta=\frac{\pi}{2}$, that is, if and only if $\gamma_\theta(K,r)$ is a geodesic ray.
\end{theorem}

\begin{proof}
	Fix $0<\theta<\pi$ and $K\in\R$, and let $r>0$ be small enough so that the positively-oriented geodesic circle $C_{p_K}(r)$ has geodesic curvature $f(K)>0$. A direct inspection yields
	\[
		k_\theta(K,r) = \cos\theta\cdot f(K)
			\begin{cases}
				>0	&\text{if }0<\theta<\frac{\pi}{2}	\\
				=0	&\text{if }\theta = \frac{\pi}{2}	\\
				<0	&\text{if }\frac{\pi}{2}<\theta<\pi
			\end{cases}
	\]
	and, as Lemma \ref{lem:f_is_decreasing} provides $f'(K)<0$,
	\[
		\frac{\partial}{\partial K}\,k_\theta(K,r) = \cos\theta\cdot f'(K)
			\begin{cases}
				<0	&\text{if }0<\theta<\frac{\pi}{2}	\\
				=0	&\text{if }\theta = \frac{\pi}{2}	\\
				>0	&\text{if }\frac{\pi}{2}<\theta<\pi,
			\end{cases}
	\]
	from which
	\[
		\frac{\partial}{\partial K}\, |k_\theta(K,r)|
			\begin{cases}
				=0	&\text{if }\theta=\frac{\pi}{2}		\\
				<0	&\text{otherwise.}
			\end{cases}
	\]
\end{proof}

Note that $k_\theta(K,r)$ and $\frac{\partial}{\partial K} k_\theta(K,r)$ depend up to a sign on the orientations of $\gamma_\theta(K,r)$ and the underlying surfaces, while the quantity $\frac{\partial}{\partial K}\,|k_\theta(K,r)|$ does not.

Conceptually, we may take Corollary \ref{cor:k_is_continuously_differentiable_in_K} to imply that, for fixed $\theta$, the family of logarithmic spirals $\gamma_\theta(K,r)$ varies smoothly in $K$ in a suitably weak sense. From this perspective, Theorem \ref{thm:norm_of_variation_of_curvature} expresses the property that $\gamma_\theta(K,r)$ ``reduces tension'' or ``uncoils'' about its center as $K$ increases.

\bibliographystyle{plain}
\bibliography{spirals}

\end{document}